\documentclass[a4paper,fleqn]{cas-sc}

\usepackage{amsmath,amsfonts,amssymb,expl3,xparse,etoolbox,xspace,xstring,footmisc}
\usepackage{booktabs,makecell,multirow,array,colortbl,dcolumn,stfloats}
\usepackage{latexsym,amscd,amsfonts,enumerate,supertabular}
\usepackage{tabularx}

\usepackage{bigstrut}
\usepackage{graphicx}
\usepackage{epsfig}

\usepackage{float}
\usepackage{hyperref}
\usepackage{placeins}

\usepackage{cite}

\usepackage{amsthm}

\numberwithin{equation}{section}
\newtheorem{lemma}{Lemma}
\numberwithin{lemma}{section}
\newtheorem{theorem}{Theorem}
\numberwithin{theorem}{section}
\newtheorem{remark}{Remark}
\numberwithin{remark}{section}
\newtheorem{example}{Example}
\numberwithin{example}{section}

\usepackage{pstricks,pst-node,pst-tree}
\usepackage{tikz}
\usetikzlibrary{fit,shapes}
\usepackage{pifont}

\def\blacknode{\hbox{\small\textbullet}}

\newcommand*\cir[1]{\tikz[baseline=(char.base)]{
            \node[circle,draw=black,inner sep=0.5pt, text width=2mm] (char) { \small $\hskip 0.5pt{#1}$};}}

\tikzset{
treenode/.style = {inner sep=0.2pt, text centered, font=\rmfamily},
arn_b/.style = {treenode, circle, white, draw=black, fill=black, text width=1.06mm},
arn_w/.style = {treenode, circle, black, draw=black, text width=1.9mm},
w/.style = {treenode, circle, black, draw=black, text width=1.3mm},
textnode/.style = {text width=2mm},
}

\definecolor{darkmagenta}{rgb}{0.55, 0.0, 0.55}

\newcommand{\q}{\quad}
\newcommand{\ee}{{\rm e}\hspace{1pt}}

\begin{document}

\shorttitle{V.T. Luan and T. Alhsmy: Sixth-order exponential Runge--Kutta methods}

%% Short author
%\shortauthors{V.T. Luan et~al.}

\title [mode = title]{Derivation of sixth-order exponential Runge--Kutta methods for stiff systems}                      


\tnotetext[1]{V.T. Luan was partially supported by NSF awards DMS--2012022 and DMS--2309821.}

\author[]{Vu Thai Luan}
\cormark[1]

\ead{luan@math.msstate.edu}

\author[]{Trky Alhsmy}

\affiliation[]{organization={Department of Mathematics and Statistics, Mississippi State University},
    addressline={410 Allen Hall}, 
    city={Mississippi State},
state={MS},
    postcode={39762}, 
    country={USA}}

\cortext[cor1]{Corresponding author}

\begin{abstract}
This work constructs the first-ever sixth-order exponential Runge--Kutta (ExpRK) methods for the time integration of stiff parabolic PDEs. First,  we leverage the exponential B-series theory to restate the stiff order conditions for ExpRK methods of arbitrary order based on an essential set of trees only. Then, we explicitly provide the 36 order conditions required for sixth-order methods and present convergence results. In addition, we are able to solve the 36 stiff order conditions in both their weak and strong forms, resulting in two families of sixth-order parallel stages ExpRK schemes. Interestingly, while these new schemes require a high number of stages, they can be implemented efficiently similar to the cost of a 6-stage method. Numerical experiments are given to confirm the accuracy and efficiency of the new schemes.
\end{abstract}

\begin{keywords}
\small 
Exponential RK methods
\sep
Stiff order conditions
\sep Sixth-order methods

\end{keywords}

\maketitle
%%%%%%%%%%%%%%%%%%%%%%%%%%%%%%%%
\section{Introduction}
\label{sec:introduction}
In this paper, we are concerned with time integration of stiff systems of initial value problems of the form
\begin{equation} \label{eq1.1}
u'(t)=A u(t) + g(t,u(t)) = F(t,u(t)),\q u(t_0)=u_0.
\end{equation}
These problems typically arise from spatial discretization of parabolic partial differential equations (PDEs), such as diffusion-reaction problems.
The stiffness here lies in the linear part, where $\|A\|$ exhibits a large norm or even represents an unbounded operator. We assume that the nonlinearity $g(t,u)$ satisfies a local Lipschitz condition with a moderate Lipschitz constant.

Explicit exponential Runge--Kutta (ExpRK)  methods have shown to be highly competitive among the viable time integration methods for integrating stiff systems  \eqref{eq1.1}, see, e.g.,
\cite{HO05,Dujardin2009,Dimarco2011,LO12b,LO14b,Luan2014,Li2014,LCR2020,Luan2021}. 
%\cite{HO05,HO05b,Dujardin2009,Dimarco2011,LO12b,LO14b,Luan2014,Li2014,Bhatt2017,LCR2020,Huang2019,Calandrini2020,Luan2021}.
They were constructed by approximating the true solution of  \eqref{eq1.1} represented via  the Duhamel's formula, leading to the following class of $s$-stage schemes (see \cite{LO12b,LO14b}): 
\begin{subequations} \label{eq:ExpRK}
\begin{align}
 U_{ni}&= u_n + c_i h \varphi _{1} ( c_i h A)F(t_n, u_n) +
 h \sum_{j=2}^{i-1}a_{ij}(h A) D_{nj}, \  2\leq i\leq s,  \label{eq2.1a} \\
u_{n+1}& = u_n + h \varphi _{1} ( h A)F(t_n, u_n) + h \sum_{i=2}^{s}b_{i}(h A) D_{ni}  \label{eq2.1b},
\end{align}
\end{subequations}
where 
 $
 D_{ni}= g (t_n+c_i h, U_{ni})- g(t_n, u_n ),  
$ 
 $U_{ni} \approx u(t_n+c_i h)$, $h = t_{n+1}-t_n >0$ is the time step size, and $c_i $ are nodes. The coefficients $a_{ij}(z)$ and $b_i (z)$ are usually linear combinations of the following functions  and their scaled versions
\begin{equation} \label{eq1.4}
%\varphi_{0}(z)=\ee^z, \q
\varphi_{k}(z)=\int_{0}^{1} \ee^{(1-\theta )z} \frac{\theta ^{k-1}}{(k-1)!}\,\text{d}\theta , \quad k\geq 1.
\end{equation}
In \cite{HO05},  stiff ExpRK methods with orders up to 4 were derived. Building upon this work, a new stiff order condition theory was developed in \cite{LO13}, enabling the derivation of a fifth-order ExpRK method in \cite{LO14b}. None of the previously constructed ExpRK methods, however, satisfies all the required stiff order conditions  in the strong form. In particular, many of the order conditions have been relaxed  in order to reduce complexity and minimize the number of stages when solving them.
%%%%%%%
Very recently, new stiff ExpRK methods of orders 4 and 5 were constructed  in \cite{Luan2021} by relaxing only one order condition as well as allow for parallel implementation of multiple internal stages, resulting in improved accuracy and efficiency. 
%compared to the previous methods of the same order.
Note that ExpRK methods might experience order reduction when applied to PDEs with inhomogeneous time-dependent (i.e., non-vanishing) boundary conditions. However, this issue can be avoided by employing a technique recently proposed and used in \cite{Cano2017,Cano2018,Cano2022a,Cano2022b} for both linear and nonlinear problems.

Due to the substantial growth in the number of stiff order conditions for ExpRK methods, constructing methods of order higher than 5 is nontrivial as this further increases the complexity in solving these order conditions (which involve matrix functions). For example, it was shown in \cite[Table 5.2]{LO13} that the number of stiff order conditions required for ExpRK methods of order 6 is 36, compared to 9 and 16 conditions for fourth- and fifth-order methods, respectively. 

In this work, we leverage the exponential B-series theory  \cite{LO13} to explicitly derive the set of 36 stiff order conditions required for ExpRK methods of order 6. Notably, we not only achieve the solution to these 36 conditions by weakening just one condition, but also successfully satisfy all the stiff order conditions in their strong forms, i.e., without weakening any conditions. 
As a result, we derive the first-ever sixth-order constructed stiff ExpRK methods. 
Additionally, while these new schemes require a high number of stages, they can be implemented efficiently similar to a 6-stage ExpRK method. This is possible because they are designed to have multiple independent internal stages, allowing for simultaneous or parallel implementation.

The organization of the paper is as follows: In Section~\ref{sec2:oc}, we present the stiff order conditions for ExpRK methods of any order based on an essential set of trees only. Additionally, we explicitly provide the 36 order conditions required for stiff ExpRK methods of order 6. Section~\ref{sec3:convergence} presents convergence results for ExpRK methods of order 6. With this, in Section~\ref{sec4:derivation} we derive the first two families of stiff ExpRK methods of order 6, which can be implemented as the cost of a 6-stage method. Finally, a numerical example is presented in Section~\ref{sec:experiments} to illustrate the theoretical results. The primary contributions of this work are  the new stiff order conditions presented in Table~\ref{Table1} and the construction of the two families of sixth-order stiff ExpRK methods, $\mathtt{ExpRK6s15}$ and  $\mathtt{ExpRK6s16}$.

%% SECTION 2: ------------------------------------------------------------------
\section{Stiff order conditions for ExpRK methods of order 6}
\label{sec2:oc}
We begin  by denoting and introducing the following sets of trees:
%%%%%%%%%%%%%%%%%%%%%%%%%%%
\begin{equation*}\label{eq:Tu}
T_u= \{ 
\begin{tikzpicture}
\node [w]{};
\end{tikzpicture}
\}
\cup
\{\tau=\cir{k}\colon k=2,3,\ldots  \},
\  \text{in which  
\begin{tikzpicture}
\node [w]{};
\end{tikzpicture}
represents  $u'(t)$, and  \cir{k}
represents $u^{(k)}(t)$ ($k \ge 2$).
}
\end{equation*}
%%%%%%%%%%%%%%%%%%%%%%%%%%%
$T_g$ represents for all the terms involving the time derivatives of $g(u(t))$ and its elementary differentials, which can be defined recursively as the smallest set of trees, satisfying the property
$
\text{if $\tau_1,\ldots, \tau_m \in T_u \cup T_g$, then $\tau=[\tau_1,\ldots, \tau_m ] \in T_g$.}
$
Here, $[, \cdot ,]$ is the commonly used notation for an operation that combines a finite number of trees $\tau_1,\ldots, \tau_m$ by connecting them to a new shared black root node \blacknode, thereby producing a new tree denoted as $\tau=[\tau_1,\ldots, \tau_m]$.
With this, one can write
\begin{equation}\label{eq:Tg}
T_g= \{ 
\tau=[\tau_1,\ldots, \tau_m ]\colon \tau_i \in T_u \cup T_g, \ i=1,\ldots,m
\}.
\end{equation}
Next, we consider two disjoint subsets of $T_g$, denoted by $T_1$ and $T_2$, which are given as:
%%%%%%%%%%%%%%%%%%%%%%%%%%%
%%%%%%%%%%%%
\begin{equation}\label{eq:T1}
%\begin{aligned}
T_1 =
 \{ 
\tau= [\tau_1,\ldots, \tau_m ]\colon 
\tau_i = 
\begin{tikzpicture}
\node [w]{};
\end{tikzpicture}, \ i=1,\ldots,m
\}
=
 \{ 
\tau=
[
\underbrace{
\begin{tikzpicture}
\node [w]{};
\end{tikzpicture}
,\ldots, 
\begin{tikzpicture}
\node [w]{};
\end{tikzpicture} 
}_{m \ \text{times}}
]: m \ge 1 
\}
 =\Big\{
\hspace*{0.1cm}
\raisebox{-2mm}{
\begin{tikzpicture}[-, grow=up, level 1/.style={sibling distance=4mm,level distance=3mm}]
\node [arn_b] {} child{ node [w]{}};
\end{tikzpicture}}
\hspace*{0.1cm},\hspace*{0.1cm}
\raisebox{-2mm}{
\begin{tikzpicture}[-, grow=up, level 1/.style={sibling distance=4mm,level distance=3mm}]
\node [arn_b] {} child{ node [w] {}} child{  node [w] {}};
\end{tikzpicture}}
\hspace*{0.1cm},\hspace*{0.1cm}
\raisebox{-2mm}{
\begin{tikzpicture}[-, grow=up, level 1/.style={sibling distance=4mm,level distance=3mm}]
\node [arn_b] {} child{ node [w] {}} child{  node [w] {}} child{  node [w] {}};
\end{tikzpicture}}
\hspace*{0.1cm},\hspace*{0.1cm}
\raisebox{-2mm}{
\begin{tikzpicture}[-, grow=up, level 1/.style={sibling distance=4mm,level distance=3mm}]
\node [arn_b] {} child{ node [w] {}} child{  node [w] {}} child{ node [w] {}} child{ node [w] {}};
\end{tikzpicture}},
\ldots
\Big\},
\end{equation}
that is used to represent for $g^{(m)}(u(t))\big(\underbrace{u'(t),\ldots,u'(t)}_{m \ \text{times}}\big)$, $ m \ge 1$, and
\begin{equation}\label{eq:T2}
T_2=
\{ \tau=[\tau_1,\ldots, \tau_m ]\colon \tau_i \in  
\{ 
\begin{tikzpicture}
\node [w]{};
\end{tikzpicture}
\}
\cup 
 T_1 \cup T_2, \ i=1,\ldots,m \}
= \Big\{
\begin{tikzpicture}[-, grow=up, level 1/.style={sibling distance=4mm,level distance=3mm}]
\node [arn_b] {} child{ node [arn_b]{} child{ node [w] {}}};
\end{tikzpicture}
\hspace*{0.2cm},\hspace*{0.1cm}
%%%%%%%%%%%%%%%%
\begin{tikzpicture}[-, grow=up, level 1/.style= {sibling distance=4mm,level distance=3mm} ]
\node [arn_b] {} child{ node [arn_b]{} child{ node [w] {}} child{ node [w] {}}};
\end{tikzpicture}
\hspace*{0.1cm},\hspace*{0.2cm}
%%%%%%%%%%%%%%%%
\begin{tikzpicture}[-, grow=up, level 1/.style={sibling distance=4mm,level distance=3mm}]
\node [arn_b] {}  child{  node [arn_b]{} child{ node [w] {}}} child{ node [w] {}};
\end{tikzpicture}
\hspace*{0.2cm},\hspace*{0.1cm}
%%%%%%%%%%%%%%%%
\begin{tikzpicture}[-, grow=up, level 1/.style={sibling distance=4mm,level distance=3mm}]
\node [arn_b] {} child{ node [arn_b]{} child{node [arn_b]{} child{node [w] {}}}};
\end{tikzpicture}
\hspace*{0.1cm},
\begin{tikzpicture}[-, grow=up, level 1/.style={sibling distance=4mm,level distance=3mm}]
\node [arn_b] {}  child{ node [arn_b]{}  child{ node [w] {}}  child{  node [w] {}} child{  node [w] {}} };
\end{tikzpicture}
\hspace*{0.1cm},\hspace*{0.2cm}
%%%%%%%%%%%%%%%%%%%
\ldots
\Big\}.
\end{equation}
For a tree  $\tau \in T_u \cup T_g$, we also use the standard notions such as the order of $\tau$, denoted by $\vert \tau \vert$, and its symmetry coefficient (see, e.g.,  \cite{Hairer1987}), denoted by $\sigma(\tau)$, which are defined recursively by
%%%%%%%%%%%%%%%
%%%%%%%%%%%%%%%
\begin{equation}\label{eq3.7}
\vert \tau \vert =
\begin{cases}
1, & 
\tau= 
\begin{tikzpicture}
\node [w]{};
\end{tikzpicture}
 \in T_u,  
\\
k, & 
\tau =\cir{k}  \in T_u, 
\\
1+\sum_{i=1}^{m} \vert \tau_i \vert, & \tau=[\tau_1,\ldots, \tau_m ] \in T_g,
\end{cases}
\end{equation}
%$\sigma(\tau)$ denotes the symmetry coefficient of tree $\tau\in T$, which is defined recursively by
\begin{equation}\label{eq3.8}
\sigma(\tau) =
\begin{cases}
\vert \tau \vert !, &\tau \in T_u,  \\
\prod_{i=1}^{k} n_{i}! \,\sigma(\tau_i)^{n_i},&  \tau=[\tau^{n_1}_1,\ldots, \tau^{n_k}_k ] \in  T_g,
\end{cases}
\end{equation}
where the notation 
$[\tau^{n_1}_1,\ldots, \tau^{n_k}_k ]$ stands for a tree $\tau = [\tau_1,\ldots, \tau_m]$ which has $k$ distinct branches among $\tau_1,\ldots, \tau_m $, say $\tau_1,\ldots, \tau_k$, corresponding to the number of occurrences $n_1,\ldots, n_k$, respectively, with $n_1+\ldots +n_k = m$.
%%%%%%%%%%%%%%%

Using these set of trees and notations, below we present a simplified version of the main result of \cite[Theorem 5.1]{LO13} which is sufficient to derive stiff order conditions for explicit ExpRK methods of arbitrary order.

First, we note that since ExpRK methods are invariant under the transformation of \eqref{eq1.1} to its corresponding autonomous version, i.e., replacing $g(t,u)$ by $g(u)$ (see \cite{LO14b}), the order conditions hold the same for both formats. Therefore, for the sake of  simplicity in presenting the derivation of order conditions, one can consider \eqref{eq1.1} with $g(t,u)  \equiv g(u)$.
%%%%%%%%%%%%%%%%%%%%%%%%%%%%%%%%%
\begin{lemma}\label{Lemma2.1}
Assuming that $A$ is the infinitesimal generator of an analytical semigroup $\ee^{tA}$ on a Banach space $X$ and the nonlinearity  $g\colon X \to X$ is sufficiently often Fr\'echet differentiable in a strip along the exact solution, and that  $u\colon [t_0, t_{\text{end}}]\to X$ is sufficiently smooth with derivatives in $X$ (all occurring derivatives are assumed to be uniformly bounded). Then, the required order conditions for stiff ExpRK methods \eqref{eq:ExpRK} of order $p$ applied to \eqref{eq1.1} are 
%%%%%%%%%%%%%%%%%%%%%%%
\begin{subequations}\label{eq:ocExpRK}
\begin{align}
\sum_{i=2}^{s}b_i(hA) \frac{c_i^{|\tau|-1}}{(|\tau|-1)!}&= \varphi_{|\tau|}(hA) \q \text{for all} \ \tau \in T_1 \ \text{with} \ |\tau|\leq p,  \label{eq2.4a}\\
\sum_{i=2}^{s}b_i (hA) g^{(m)}(u)\big( \mathcal{S}_i(\tau_1)(u),\ldots,\mathcal{S}_i(\tau_m)(u) \big)&=0 \q \text{for all} \ \tau=[\tau_1,\ldots, \tau_m ] \in T_2 \ \text{with} \ |\tau|\leq p,\label{eq2.4b}
\end{align}
\end{subequations}
where $ \mathcal{S}_i(\tau_k)(u)$ are the elementary differentials given recursively as 
\begin{equation}\label{eq:Stau}
\mathcal{S}_i(\tau_k)(u)=
\begin{cases}
c_i u'(t), & \text{if} \ 
\tau_k = 
\begin{tikzpicture}
\node [w]{};
\end{tikzpicture}
 \in T_u  
 \\[2mm]
 \psi_{\ell +1,i} (hA)\,g^{(\ell)}(u)\big(\underbrace{u'(t),\ldots, u'(t)}_{\ell \ \text{times}} \big), &   \text{if} \ 
\tau_k= [
\underbrace{
\begin{tikzpicture}
\node [w]{};
\end{tikzpicture}
,\ldots, 
\begin{tikzpicture}
\node [w]{};
\end{tikzpicture} 
}_{\ell \ \text{times}}
]
 \in T_1 \\[2mm]
\tfrac{\sigma(\tau_{k1})\ldots\sigma(\tau_{k\ell})}{\sigma(\tau_k)}\ \sum_{j=2}^{i-1}a_{ij}(hA) g^{(\ell)}(u)\big( \mathcal{S}_j(\tau_{k1})(u),\ldots,\mathcal{S}_j(\tau_{k\ell})(u) \big), &  \text{if} \ \tau_k= [\tau_{k1},\ldots, \tau_{k\ell} ] \in T_2
\end{cases}
\end{equation}
with
\begin{equation}\label{eq2.9}
\psi_{q,i} (h A)= \sum_{k=2}^{i-1}a_{ik} (hA) \frac{c^{q-1}_{k}}{(q-1)!}- c^{q}_{i}\varphi_{q} (c_i h A).
\end{equation}
\end{lemma}
%%%%%%%%%%%%%%%%%%%%%%%%%%%%%
\begin{proof}
Considering one-step integration \eqref{eq:ExpRK} from $t_n$ to $t_{n+1} = t_n + h$, we denote 
$\tilde{e}_{n+1}=\hat{u}_{n+1}- u(t_{n+1})$ as the local error of ExpRK methods, where $\hat{u}_{n+1}$ represents the numerical solution obtained using the initial value $\tilde{u}_n = u(t_n)$. 
Using the results of \cite[Section 5.2]{LO13}, it is clear that the set of trees 
$ \{ 
\begin{tikzpicture}
\node [w]{};
\end{tikzpicture}
\}
\cup 
 T_1 \cup T_2
 $ 
 introduced earlier constitutes a minimal set of trees required for deriving the stiff order conditions \eqref{eq:ocExpRK}.
Specifically, by using \cite[eqs. (5.1)]{LO13}, one can express the local error $\tilde{e}_{n+1}$ as
\begin{equation}\label{eq:local_err}
\begin{aligned}
\tilde{e}_{n+1} 
& =\sum_{ \tau=[\tau_1,\ldots, \tau_m ] \in T_1} h^{|\tau|}
\Big(\sum_{i=2}^{s}b_i(hA) \tfrac{c_i^{|\tau|-1}}{(|\tau|-1)!} - \varphi_{|\tau|}(hA) \Big)
 %\tfrac{(\vert \tau \vert -1)!}{\sigma(\tau)}  
g^{(m)}(\tilde{u}_n))  
\big(\underbrace{u'(t_n),\cdots,u'(t_n)}_{m \ \text{times} }\big)
\\
%& + \sum_{\tau=[\tau_1,\ldots, \tau_m ] \in T_g | (T_1 \cup T_2) } h^{|\tau|}
%\Big(\sum_{i=2}^{s}b_i(hA) \tfrac{c_i^{|\tau|-1}}{(|\tau|-1)!} - \varphi_{|\tau|}(hA) \Big)
% \tfrac{(\vert \tau \vert -1)!}{\sigma(\tau)}  
%g^{(m)}(\tilde{u}_n)) \big( u^{(|\tau_1|)}(t),\ldots, u^{(|\tau_m|)}(t) \big)    \\
%%%%
& +\sum_{ \tau=[\tau_1,\ldots, \tau_m ] \in T_2 } h^{|\tau|} \tfrac{\sigma(\tau_1)\ldots\sigma(\tau_m)}{\sigma(\tau)}\sum_{i=2}^{s}b_i(h A) g^{(m)}(\tilde{u}_n)\big( \mathcal{S}_i(\tau_1)(\tilde{u}_n),\ldots,\mathcal{S}_i(\tau_m)(\tilde{u}_n) \big)\\
& + 
\underbrace{
\sum_{\tau \in T_g | \ (T_1 \cup T_2) } h^{|\tau|} 
\text{{\footnotesize \big(same weighted coefficients as trees in $T_1$ or $T_2$\big)  $\times$ \big(corresponding elementary differentials\big)}}.
}_{\text{\emph{{\footnotesize remainder terms including the trees which have the same order conditions with trees in $T_1 \cup T_2$}}}
}
\end{aligned}
\end{equation}
%%%%%%%%%%%%%%%%%%%%%%%%%%%
%\begin{equation}\label{eq:Stau}
%\mathcal{S}_i(\tau_k)(u)=
%\begin{cases}
%\tfrac{c_i^{|\tau_k |}}{|\tau_k|!}\; u^{(|\tau_k |)}(t), & \text{if} \ \tau_k\in T_0 \\[2mm]
%\tfrac{(\vert \tau_k\vert -1)!}{\sigma(\tau_k)}  \psi_{|\tau_k|,i} (hA)\,g^{(\ell)}(u)\big( u^{(|\tau_{k1}|)}(t),\ldots, u^{(|\tau_{k\ell}|)}(t) \big), &   \text{if} \ \tau_k= [\tau_{k1},\ldots, \tau_{k\ell} ]  \in T_1 \\[2mm]
%\tfrac{\sigma(\tau_{k1})\ldots\sigma(\tau_{k\ell})}{\sigma(\tau_k)}\ \sum_{j=2}^{i-1}a_{ij}(hA) g^{(\ell)}(u)\big( \mathcal{S}_j(\tau_{k1})(u),\ldots,\mathcal{S}_j(\tau_{k\ell})(u) \big), &  \text{if} \ \tau_k= [\tau_{k1},\ldots, \tau_{k\ell} ] \in T_2
%\end{cases}
%\end{equation}
The assumption on $A$ implies that the semigroup $\ee^{tA}$ can be bounded uniformly (see, e.g.,  \cite{H81,PAZY83}) and thus $b_i (hA), a_{ij}(hA)$, and  $\psi_{j} (hA)$ are  uniformly bounded as well. Together with the remaining assumptions of Lemma~\ref{Lemma2.1}, one can then truncate \eqref{eq:local_err} up to any stiff order $p =|\tau|$ which requires the conditions \eqref{eq:ocExpRK}. 
%The conclusion of Lemma~\ref{Lemma2.1} can be now obtained directly by applying  \cite[Theorem 5.1]{LO13} to $T$.
\end{proof}
By using Lemma~\ref{Lemma2.1}, we can explicitly provide the 36 order conditions required for stiff ExpRK methods of order 6 by only drawing trees in $T_1 \cup T_2$ up to order 6 and using  \eqref{eq:ocExpRK}, see Table~\ref{Table1} below. There, we note that the 20 order conditions from No. 17 to 36  are new. 
%(the first 16 conditions have been already derived in \cite{LO12b}).  
With this result, we now present convergence results for ExpRK methods of order 6.

%%------------- Table of stiff order conditions for ExpRK methods -------------
%\FloatBarrier
\begin{table}[H]
\begin{center}
\caption{Order trees and stiff order conditions for explicit ExpRK methods up to order~6. The variables $Z$, $J$, $K$, $L$ denote arbitrary square matrices, and $B$ an arbitrary bilinear mapping of appropriate
dimensions. The functions $\psi_{q,i}$ are defined in \eqref{eq2.9}.}\label{Table1}
\includegraphics[scale=0.6]{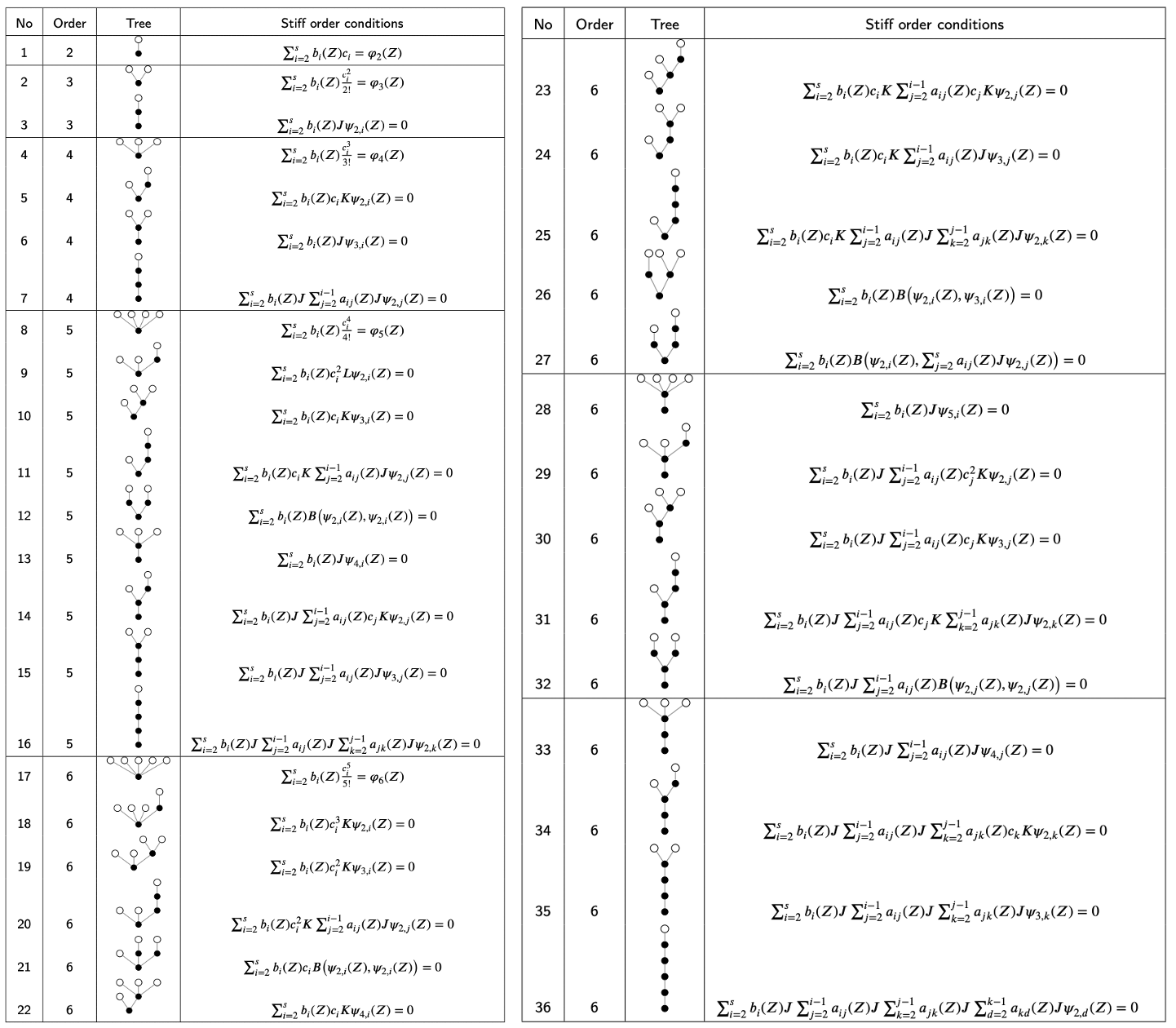}
\end{center}
\vspace{-4mm}
\end{table}
\vspace{-3mm}

%% SECTION 3: ------------------------------------------------------------------
\section{Convergence results for sixth-order stiff ExpRK methods}
\label{sec3:convergence}
Throughout this section, we denote $C$ as a generic constant that may have different values at different occurrences.
%%--THEOREM 3.1--------
\begin{theorem}\label{Theorem3.1}
Under the assumptions of Lemma~\ref{Lemma2.1}, an explicit ExpRK method \eqref{eq:ExpRK} applied to the IVP \eqref{eq1.1}, that fulfills all the 36 order conditions of Table~\ref{Table1} either strictly in the strong sense (without relaxing any conditions) or with the exception that only condition No. 17 holds in a weakened form  $\sum_{i=2}^{s}b_{i}(0)\frac{c_i^5}{5!} = \varphi_6(0) =\frac{1}{6!}$, converges with order 6. 
In particular, its global error  $ e_{n} = u_{n} - u(t_{n})$ satisfies the bound 
%the numerical solution $u_n$ satisfies the global error bound
\begin{equation}\label{eq3.1}
\|e_n\| = \| u_n -u(t_n)\|\leq C h^6
\end{equation}
uniformly on time intervals \ $t_0 \leq  t_n =t_0+nh \leq  t_{\text{end}}$ with a constant $C$ that depends on $t_{\text{end}}-t_0$, but is independent of $n$ and $h$.
\end{theorem}
\begin{proof}
 It was shown in \cite[Sect. 4.3]{LO14b} that 
\begin{equation} \label{eq3.2}
e_{n}=u_n -u(t_n)= h\sum_{j=0}^{n-1} \ee^{(n-j)hA} \mathcal{K}_j (e_j)e_j + \sum_{j=0}^{n-1} \ee^{jhA}\tilde{e}_{n-j},
\end{equation}
where 
the operators $\mathcal{K}_{j}$ depend on $e_j$ (as well as $b_i (hA), a_{ij}(hA),   g^{(m)}(.)(.,\cdots,.)$), which can be shown to be bounded on $X$ under the given assumptions (see \cite[Lemma 4.1.]{LO14b}).
Let us now consider the first scenario where all 36 stiff order conditions for ExpRK methods of order 6 are strictly satisfied. 
In this case, it is clear from \eqref{eq:local_err} that, for $ |\tau| =6$, the local error $\tilde{e}_{n+1}=\mathcal{O}(h^{7})$.
%applying \cite[Theorem 5.1]{LO13} for $p = 6$ immediately implies that $\tilde{e}_{n+1}=\mathcal{O}(h^{7})$.
The assumption on $A$ implies
$
\|\ee^{tA}\|_{X\leftarrow X}\leq C, \ t\geq 0.
$
%(the constant $C$ can be chosen uniformly).
Thefore, one can bound $e_n$ in \eqref{eq3.2} as 
\begin{equation} \label{eq3.3}
\|e_{n}\| \le h\sum_{j=0}^{n-1} C \|e_j \|+ \sum_{j=0}^{n-1}C h^7.
\end{equation}
An application of a discrete Gronwall lemma to \eqref{eq3.3} shows the error bound  \eqref{eq3.1} at once. 

Finally, we consider the second scenario wherein condition No. 17 is relaxed to be fulfilled with   $A =0$ only 
(i.e, $\psi_{6} (0) =0$), while all other conditions are strictly satisfied in the strong sense. 
In this case,  \eqref{eq:local_err} can be simplified to 
\begin{equation} \label{eq3.4}
\tilde{e}_{n+1}=h^6 \big( \psi_{6} ( h A)-\psi_{6} (0)\big)  g^{(5)}(u(t_n))\big(u'(t_n), u'(t_n),u'(t_n), u'(t_n),u'(t_n)\big)
+\mathcal{O}(h^{7}).
\end{equation}
%where $D^{5} g(t_n, u(t_n))$ is the fifth-order partial Fr\'echet  derivative  of $g(t, u)$ (with respect to $u$) evaluated at $u(t_n)$.
Clearly, one can show that there exists a bounded operator $\tilde{\psi} (hA)$ such that  $ \psi_{6} ( h A)-\psi_{6} (0)=\widehat{\psi} (hA)hA$. 
Using this, inserting \eqref{eq3.4} into  \eqref{eq3.2} and employing the parabolic smoothing property of the analytical semigroup (implying the bound 
$\left \|h A \sum_{j=1}^{n}\ee^{j h A} \right\|_{X\leftarrow X}  \leq C$),  one gets 
$
%\begin{equation} \label{eq3.6}
\|e_{n}\| \le h\sum_{j=0}^{n-1} C \|e_j \|+ Ch^6 + \sum_{j=0}^{n-1}C h^7,
%\end{equation}
$
which again proves  \eqref{eq3.1} as done for  \eqref{eq3.3}.
\end{proof}
%%%%%%%%%%%%%%%%%%%%%%%%%%%%%%%%%%
\begin{remark}\label{remark3.1} 
Similar to the convergence proof for ExpRK methods of order 5 (see \cite[Thm. 4.1]{LO14b}), we note that the result stated in Theorem~\ref{Theorem3.1} (the error bound \eqref{eq3.1}) also remains valid if all the order conditions of order 6 (from No. 17 to 36) are satisfied in a weakened form with $b_i (0)$ substituted for $b_i (Z)$. However, we only considered weakening condition No. 17 due to two primary advantages (similar to the approach presented in \cite{Luan2021} for fourth-and fifth-order ExpRK methods). Firstly, No. 17 depends solely on $b_i (hA), c_i$ and thus relaxing it leads to a scalar algebraic equation for $c_i$ (which can be solved easily). Secondly, satisfying conditions from No. 18 to 36 stricly in the strong sense (which, in addition, depend also on $a_{ij}(hA)$)  not only offers better stability when solving stiff problems but also enables the construction of multiple parallel internal stages $U_{ni}$, thereby leading to more efficient schemes (if otherwise, i.e., weakening them results in methods where each internal stage depends on preceding stages, restricting them to sequential implemenation, see e.g., the construction of ExpRK methods of order 5 in \cite{LO14b}). 
\end{remark}
%%%%%%%%%%%%%%%%%%%%%%%%%%%%%%%%%%
%% SECTION 4: ------------------------------------------------------------------
\section{Derivation of sixth-order stiff ExpRK methods}
\label{sec4:derivation}
Based on Theorem~\ref{Theorem3.1}, we now derive two families of sixth-order ExpRK methods.
%which fulfill all the 36 order conditions of Table~\ref{Table1} in the strong sense or with the exception that condition No. 17 holds in a weakened form. 
While there exist many solutions, our aim is to construct efficient schemes which allow groups of \emph{parallel stages}, i.e., internal stages that are independent of each other.

In  \cite{Luan2021}, it has been shown that solving the first 16 conditions for fifth-order parallel stages ExpRK methods requires $s=10$ stages (with relaxing one condition, namely, No. 8) and $s=11$ stages (without relaxing any conditions). 
To satisfy the 20 additional conditions (from No. 17 to 36) for order 6, a value of $s \geqslant 11$ is thus necessary.
Applying a similar approach to the order conditions outlined in \cite{Luan2021}, it can be verified (but very tedious) that using $s=11, 12, 13, 14$ is not sufficient to fulfill all the order conditions of Table~\ref{Table1} in the context of Theorem~\ref{Theorem3.1}. We omit the details.
 On the other hand,  it is indeed possible  to satisfy all the conditions in the strong sense with $s=16$ and if condition No. 17 is relaxed, one can use $s=15$.
We will first show the latter case.

For convenience, we use the following abbreviations $b_i = b_i (hA),  a_{ij} = a_{ij}(hA),  \varphi_i = \varphi_i (h A), \varphi_{j,i} = \varphi_j (c_i h A) $.
  
%%%%%%%%%%%%%%%%%%%%%%%%%%%%%%%%%%%%%%%
\subsection{A family of sixth-order parallel-stage ExpRK schemes with $s=15$} 
In this case, we assume that condition No. 17 can be relaxed to  $\sum_{i=2}^{15}b_{i}(0)\frac{c_i^5}{5!} = \varphi_6(0) =\frac{1}{6!}$ and will solve for the remaining 35 conditions in Table~\ref{Table1}.  First, using  \eqref{eq2.9} for $j = 2, \ldots, 5$ and the fact that  $\{\varphi_j \}$ are linearly independent, one can show that 
 $\psi_{j,2}=-c_2^j \varphi_{j,2} \neq 0 $, and that $\psi_{j,3}$, $\psi_{j,4}, \psi_{j,5}$ cannot be zero for all $j = 2, \ldots, 5$. 
Thus, in order for later satisfy  consider conditions 3, 5, 6, 9,  10, 12, 13,  18, 19, 21, 22, 26, and 28 (which involve the matrix functions $\psi_{j,i} (Z)$, $j = 2, \ldots, 5$) in the strong sense (with arbitrary square matrices $Z, J, K, L$), it strongly suggests to choose $b_{i}=0 \ (i=2, \ldots, 11)$.
Using this, we next solve conditions 1, 2, 4, 8, and 17 (in weakened form) 
and get the unique solution
\begin{equation} \label{eq4.3}
b_{i} = \frac{-c_{j} c_{k} c_{l} \varphi_{2}
+ 2 (c_{j} c_{k} +c_{j} c_{l}+ c_{k} c_{l})  \varphi_{3}
-6 (c_{j} +c_{k} +c_{l} ) \varphi_{4}
+24 \varphi_{5}}
{c_{i}(c_{i}-c_{j})(c_{i}-c_{k})(c_{i}-c_{l})}, \q i=12,13,14,15
\end{equation}
where 
$j,  k, l \in\{12,13,14,15\}, j \neq k \neq l \neq i $ and nodes $c_{12}, c_{13}, c_{14}, c_{15}>0$  are distinct  and satisfy the following relation
\begin{equation} \label{eq4.4}
\frac{1}{5}\sum_{i=12}^{15} c_i - \frac{1}{4} \sum_{\substack{i, j =12\\ i \neq j}}^{15} c_i c_j  + \frac{1}{3}  \sum_{\substack{i, j, k =12\\ i \neq j \neq k}}^{15} c_i c_j c_k
-\frac{1}{2} \prod_{i=12}^{15} c_i = \frac{1}{6}.
\end{equation} 
Since $ b_{12},b_{13},b_{14},b_{15} \neq 0$, one has to enforce $\psi_{i,j}= 0$  for $i=2,\ldots,5; j = 12, \ldots, 15$
to strictly satisfy conditions 3, 5, 6, 9,  10, 12, 13,  18, 19, 21, 22, 26, and 28.
Using this, requiring $\psi_{i, j}=0$ for $i= 2, 3, 4; j=8,\ldots,11$, and choosing $a_{i j}=0$ for $i =12,\ldots, 15; \ j=2,\ldots,7$,
one can fulfill conditions 7, 15 and 33  
 in the strong form and have the stages $\{U_{nj}\}_{j=12,\ldots,15}$ 
 that are independent of  $\{U_{nj}\}_{j=2,\ldots,7}$. 
With all the requirements above,  conditions 11, 14, 20, 23, 24, 27, 29, 30, 32  are automatically fulfilled.  

When solving the linear system $\psi_{2,j}= 0$ ($j = 12, \ldots, 15$), several $a_{ij}$ can be taken as free parameters and we choose 
$a_{i j}=0$ for $i =13, 14, 15; \ j=12, 13, 14$ 
in order to have the four parallel stages  $\{U_{nj}\}_{j=12,\ldots,15}$, 
and thus the remaining $a_{ij}$ appearing in this system can be then uniquely solved as
\begin{equation}\label{eq4.11}
 a_{ij} 
=\frac{-c_{i}^{2}c_{d} c_{k} c_{l} \varphi_{2,i}
+2 c_{i}^{3}(c_{d} c_{k} +c_{d} c_{l}+ c_{k} c_{l}) \varphi_{3,i}
-6 c_{i}^{4}(c_{d} +c_{k} +c_{l}  )\varphi_{4,i}
+24 c_{i}^{5}\varphi_{5,i}}{c_{j}(c_{j}-c_{d})(c_{j}-c_{k})(c_{j}-c_{l})},  \q i=12,13,14,15
\end{equation}
with
$j, k, l, d  \in\{8,9,10,11\}, j \neq k \neq l \neq d $ and $c_8, c_9,c_{10},c_{11}>0$ are distinct nodes.

Next, using all the above constraints, one can show that conditions 16 and 35 can be fulfilled in their strong forms if requiring
 $a_{i j}=0$ for $i =8,\ldots, 11; \ j=2, 3, 4$
 and 
$ \psi_{2, j}=\psi_{3, j}=0 \ (j=5,6,7)$.
Similarly, we then choose free parameters
 $a_{i j}=0$ for $i =9,10, 11; \ j=8, 9, 10$
to have four parallel stages  $\{U_{nj}\}_{j=8,\ldots,11}$. 
With this, conditions 25, 31 and 34  are now automatically fulfilled.
The system $\psi_{i, j}=0$ for $i= 2, 3, 4; j=8,\ldots,11$ is then solved with the unique solution
\begin{equation}\label{eq4.16}
a_{i j}=\frac{c_{i}^{2} c_{k} c_{l} \varphi_{2, i}-2 c_{i}^{3}(c_{k}+c_{l}) \varphi_{3, i} + 6 c_{i}^{4} \varphi_{4, i}}{c_{j}(c_{j}-c_{l})(c_{j}-c_{k})}, \q i=8,9,10,11; j, k, l \in\{5,6,7\}, j \neq k \neq l, \ \text{and}\  c_5, c_6,c_7 \ \text{are distinct}.
\end{equation}
Finally, we solve the only remaining condition 36 by taking into account 
all the above findings. It can be satisfied in the strong sense by enforcing $a_{k 2}=0 \ ( k =5,6,7)$ and $\psi_{2, 3}=\psi_{2, 4}=0$. This linear system and the above $ \psi_{2, j}=\psi_{3, j}=0 \ (j=5,6,7)$ can be solved by requiring $a_{43} = a_{65}=a_{75}=a_{76}=0$ to have two more groups of parallel stages $\{U_{nj}\}_{j=5,6,7}$ and $\{U_{nj}\}_{j=3,4}$, resulting in
\begin{equation}\label{eq4.23}
a_{i j}=\frac{-c_{i}^{2} c_{k} \varphi_{2, i}+2 c_{i}^{3} \varphi_{3, i}}{c_{j}(c_{j}-c_{k})}, \ \text{for} \ i=5,6,7  ;\ j, k \in\{3,4\},  \ j \neq k 
 \ \text{and} \ a_{ij}=\frac{c_i^2}{c_j} \varphi_{j,i} \ \text{for} \  i=3, 4; j =2.
\end{equation}
Putting altogether, we obtain the following family of sixth-order 15-stage stiff ExpRK methods, which will be called $\mathtt{ExpRK6s15}$:
\begin{align*}
U_{n2} &=u_{n}+\varphi_{1}(c_{2} h A)c_{2} h F(t_n,u_{n}),\\
U_{n\ell}& =u_{n}+\varphi_{1}(c_l h A)c_l h  F(t_n,u_{n})+h  a_{\ell 2} (h A)D_{n2},  &&\ell=3,4\\
U_{nm}&=u_{n}+\varphi_{1}(c_{m} h A)c_{m} h  F(t_n,u_{n})+h  (a_{m3} D_{n3}+a_{m4} D_{n4}), && m=5,6,7\\
U_{nq}&=u_{n}+\varphi_{1}(c_{q} h A) c_{q} h  F(t_n, u_{n}) + h \sum_{j=2}^{4} \varphi_{j}(c_q hA) c^{j}_{q}  \sum_{i=5}^{7}\rho_{ji} D_{ni},
&& q=8,\ldots,11\\
U_{nk}& =u_{n}+\varphi_{1}(c_{k} h A) c_{k} h  F(t_n, u_{n}) + h \sum_{j=2}^{5} \varphi_{j}(c_k hA) c^{j}_{k}  \sum_{i=8}^{11}\mu_{ji} D_{ni},
&&k=12,\ldots,15\\
u_{n+1}&=u_{n}+\varphi_{1}(h A) h  F(t_n, u_{n})
+ h \sum_{j=2}^{5} \varphi_{j}(hA)   \sum_{i=12}^{15}\mu_{ji} D_{ni},
\end{align*}
%\end{subequations*}
where
%\begin{equation}
$
\rho_{ji}= 
\begin{cases}
\tfrac{c_{k} c_{l} }{c_{i}(c_{i}-c_l)(c_{i}-c_{k})}, &  j = 2\\
\tfrac{-2(c_{k}+c_l )}{c_{i}(c_{i}-c_l)(c_{i}-c_{k})}, & j = 3 \\
\tfrac{6}{c_{i}(c_{i}-c_l)(c_{i}-c_{k})}, & j = 4
 \end{cases}
% \end{equation}
$
\ ($i, k, l \in\{5, 6, 7\}$, $i \neq k \neq l$), and
$
%\begin{equation}
\mu_{ji}= 
\begin{cases}
\tfrac{-c_{d} c_{k} c_{l}}{c_{i}(c_{i}-c_{d})(c_{i}-c_{k})(c_{i}-c_l)}, &  j = 2\\
\tfrac{2(c_{d} c_{k} +c_{d} c_{l}+ c_{k} c_{l})}{c_{i}(c_{i}-c_{d})(c_{i}-c_{k})(c_{i}-c_l)}, & j = 3 \\
\tfrac{-6(c_{d} +c_{k} +c_{l})}{c_{i}(c_{i}-c_{d})(c_{i}-c_{k})(c_{i}-c_l)}, & j = 4\\
\tfrac{24}{c_{i}(c_{i}-c_{d})(c_{i}-c_{k})(c_{i}-c_l)},  & j = 5
 \end{cases}
 %\end{equation}
 $
with $i\in\{8, 9, \ldots, 15\}$ for $ d, k, l \in\{8, 9, 10, 11\}, \ i \neq d \neq k \neq l$ . 

As observed, while $\mathtt{ExpRK6s15}$ uses $s=15$ stages, it can be implemented with the cost of a 6-stage method since it has 4 groups of parallel stages   $\{U_{nj}\}_{j=3,4}$, $\{U_{nj}\}_{j=5,6,7}$, $\{U_{nj}\}_{j=8,\ldots,11}$, and $\{U_{nj}\}_{j=12,\ldots,15}$,
which can be computed simultaneously or in parallel.\\
For our numerical experiments in Section~\ref{sec:experiments}, we take
$c_2 =c_3 =c_5= \tfrac{1}{2}, c_4 = c_9 = c_{13} = \tfrac{1}{3}, c_6 = c_{15}= \tfrac{1}{5}, c_7 = \tfrac{1}{4},  c_8 = \tfrac{18}{25},  c_{10} = c_{14} = \tfrac{3}{10},    c_{11} = \tfrac{1}{6},  c_{12} = \tfrac{90}{103}, c_{15} =  \tfrac{1}{5}$,
which satisfy the constraint \eqref{eq4.4} due to relaxing  condition  17.

%%%%%%%%%%%%%%%%%%%%%%%%%%%%%%%%%%%%%%
\subsection{A family of sixth-order parallel-stage ExpRK schemes with $s=16$} 
When $s=16$, it becomes feasible to solve all 36 stiff order conditions strictly  using a very similar approach to the $s=15$ case. We note that in this case such a restriction on the nodes $c_i$ like  \eqref{eq4.4} is no longer required. 
While we omit the specific details, we present the final result for the following family of sixth-order 16-stage stiff ExpRK methods, which will be referred to as $\mathtt{ExpRK6s16}$:
\begin{align*}
U_{n2} &=u_{n}+\varphi_{1}(c_{2} h A)c_{2} h F(t_n,u_{n}),\\
U_{n\ell}& =u_{n}+\varphi_{1}(c_l h A)c_l h  F(t_n,u_{n})+h  a_{\ell 2} (h A)D_{n2},  &&\ell=3,4\\
U_{nm}&=u_{n}+\varphi_{1}(c_{m} h A)c_{m} h  F(t_n,u_{n})+h  (a_{m3} D_{n3}+a_{m4} D_{n4}), && m=5,6,7\\
U_{nq}&=u_{n}+\varphi_{1}(c_{q} h A) c_{q} h  F(t_n, u_{n}) + h \sum_{j=2}^{4} \varphi_{j}(c_q hA) c^{j}_{q}  \sum_{i=5}^{7}\rho_{ji} D_{ni},
&& q=8,\ldots,11\\
U_{nk}& =u_{n}+\varphi_{1}(c_{k} h A) c_{k} h  F(t_n, u_{n}) + h \sum_{j=2}^{5} \varphi_{j}(c_k hA) c^{j}_{k}  \sum_{i=8}^{11}\mu_{ji} D_{ni},
&&k=12,\ldots,16\\
u_{n+1}&=u_{n}+\varphi_{1}(h A) h  F(t_n, u_{n})
+ h \sum_{j=2}^{6} \varphi_{j}(hA)   \sum_{i=12}^{16}\theta_{ji} D_{ni},
\end{align*}
%\end{subequations*}
where
%\begin{equation}
$
\rho_{ji}= 
\begin{cases}
\tfrac{c_{k} c_{l} }{c_{i}(c_{i}-c_l)(c_{i}-c_{k})}, &  j = 2\\
\tfrac{-2(c_{k}+c_l )}{c_{i}(c_{i}-c_l)(c_{i}-c_{k})}, & j = 3 \\
\tfrac{6}{c_{i}(c_{i}-c_l)(c_{i}-c_{k})}, & j = 4
 \end{cases}
 $
% \end{equation}
($i, k, l \in\{5, 6, 7\}$, $i \neq k \neq l$),  
$
%\begin{equation}
\mu_{ji}= 
\begin{cases}
\tfrac{-c_{d} c_{k} c_{l}}{c_{i}(c_{i}-c_{d})(c_{i}-c_{k})(c_{i}-c_l)}, &  j = 2\\
\tfrac{2(c_{d} c_{k} +c_{d} c_{l}+ c_{k} c_{l})}{c_{i}(c_{i}-c_{d})(c_{i}-c_{k})(c_{i}-c_l)}, & j = 3 \\
\tfrac{-6(c_{d} +c_{k} +c_{l})}{c_{i}(c_{i}-c_{d})(c_{i}-c_{k})(c_{i}-c_l)}, & j = 4\\
\tfrac{24}{c_{i}(c_{i}-c_{d})(c_{i}-c_{k})(c_{i}-c_l)},  & j = 5
 \end{cases}
 %\end{equation}
 $
($i, d, k, l \in\{8, 9, \ldots, 11\}$, \q $ i \neq d \neq k \neq l$), and 
$
%\begin{equation}
\theta_{ji}= 
\begin{cases}
\tfrac{-c_{j} c_{d} c_{k} c_{l}}  {c_{i}(c_{i}-c_{j})(c_{i}-c_{d})(c_{i}-c_{k})(c_{i}-c_{l})},   &  j = 2 \\
\tfrac{2(c_{j} c_{k} c_{d}+c_{j} c_{k}c_{l}+ c_{k} c_{d}c_{l}+c_{d} c_{l}c_{j})}  {c_{i}(c_{i}-c_{j})(c_{i}-c_{d})(c_{i}-c_{k})(c_{i}-c_{l})},  & j = 3 \\
\tfrac{-6(c_{j} c_{k}+c_{j} c_{d}+c_{j} c_{l}+c_{k} c_{d}+c_{k} c_{l}+c_{d} c_{l}  )}  {c_{i}(c_{i}-c_{j})(c_{i}-c_{d})(c_{i}-c_{k})(c_{i}-c_{l})},   & j = 4 \\
\tfrac{24(c_{j} +c_{k} +c_{l}+c_{d})}   {c_{i}(c_{i}-c_{j})(c_{i}-c_{d})(c_{i}-c_{k})(c_{i}-c_{l})},   & j = 5\\
\tfrac{-120} {c_{i}(c_{i}-c_{j})(c_{i}-c_{d})(c_{i}-c_{k})(c_{i}-c_{l})},   & j = 6
 \end{cases}
 %\end{equation}
 $ \q 
($i, j, d, k, l \in\{12, \ldots, 16 \}$, \ $ i \neq j \neq d \neq k \neq l$).

Similarly, $\mathtt{ExpRK6s16}$ also has 4 groups of parallel stages  $\{U_{nj}\}_{j=3,4}$, $\{U_{nj}\}_{j=5,6,7}$, $\{U_{nj}\}_{j=8,\ldots,11}$, and $\{U_{nj}\}_{j=12,\ldots,16}$.
%which are $\left\{U_{n 3}, U_{n 4}\right\},\left\{U_{n 5}, U_{n 6}, U_{n 7}\right\}$,  $\left\{U_{n 8}, U_{n 9}, U_{n 10}, U_{n 11}\right\}$, and $\left\{U_{n 12}, U_{n 13}, U_{n 14}, U_{n 15}, U_{n 16}\right\}$. 
This again offers a great advantage is that it can be implemented with the cost of a 6-stage method. 
For our numerical experiments, we choose  
%c2=1/2;c3=1/2;c4=1/3;c5=1/2;c6=1/5;c7=1/4;c8=1/2;c9=1/5;c10=1/4;c11=1/3;c12=1/2;c13=1/5;c14=1/4;c15=1/3;c16=1;
$c_2 =c_3 =c_5= c_8 = c_{12}=  \tfrac{1}{2}, c_4 =  c_{11} = c_{15} = \tfrac{1}{3}, c_6 = c_9 = c_{13}= \tfrac{1}{5}, c_7 = c_{10} = c_{14} = \tfrac{1}{4}, c_{16} =1$.

%% SECTION 5: ------------------------------------------------------------------
\section{Numerical experiments}
\label{sec:experiments}
% Luan added:
In this section, we demonstrate  the sharpness of the error bound  in Theorem~\ref{Theorem3.1} when applying the two newly constructed sixth-order schemes $\mathtt{ExpRK6s15}$ and  $\mathtt{ExpRK6s16}$. We also compare their efficiency when implemented with and without the simultaneous computation of parallel stages.
To implement the new schemes, we use the code \texttt{phipm\_simul\_iom.m}, see \cite{Luan2021,Luan18}.

\begin{example}\label{ex1}\rm
Consider the following semilinear parabolic  PDE  in  $X=L^2([0,1])$ (see \cite{HO05}) for $u(x,t)$,  $x\in [0,1], t\in [0,1]$, subject to homogeneous Dirichlet boundary conditions,
\begin{equation} \label{example1}
\frac{\partial u(x,t)}{\partial t}- \frac{\partial^2 u(x,t)}{\partial x^2}   =\frac{1}{1+u^{2}(x,t)}+\Phi (x,t).
\end{equation}
With a suitable choice of the source function $\Phi (x,t)$, the exact solution is $u(x,t)=x(1-x)\ee^t$.  
%This PDE fits into our framework with $X=L^2([0,1])$.
\end{example}
For a spatial discretization of \eqref{example1}, we use standard second order finite differences with $200$ grid points, leading to a very stiff system  (with $\|A\|  \approx1.6 \times10^5$) of the form  \eqref{eq1.1}.  
The resulting system is then integrated on the time interval $[0, 1]$  using constant step sizes $h=\tfrac{1}{2}, \tfrac{1}{4}, \tfrac{1}{8}, \tfrac{1}{16}, \tfrac{1}{32}$. 
The errors are measured in a discrete $L^2$ norm  at the final time $t_{\text{end}}=1$.

As seen from  Fig.~\ref{fig5.1}, 
%we plot the convergence rates (left) and the total CPU time (right) for the two new schemes $\mathtt{ExpRK6s15}$ and  $\mathtt{ExpRK6s16}$. 
the left diagram clearly shows that they fully achieve order 6 even at large time steps, thereby verifying the sharpness of our error bound  in Theorem~\ref{Theorem3.1}. With the same set of step sizes, $\mathtt{ExpRK6s16}$ exhibits slightly higher accuracy compared to $\mathtt{ExpRK6s15}$, while both schemes require similar CPU times. As anticipated, the simultaneous computation of parallel stages significantly reduces the CPU times for the same level of accuracy, as shown in the right diagram.
\begin{figure}[ht!]
\centering
\begin{tabular}{cc}
\epsfig{file=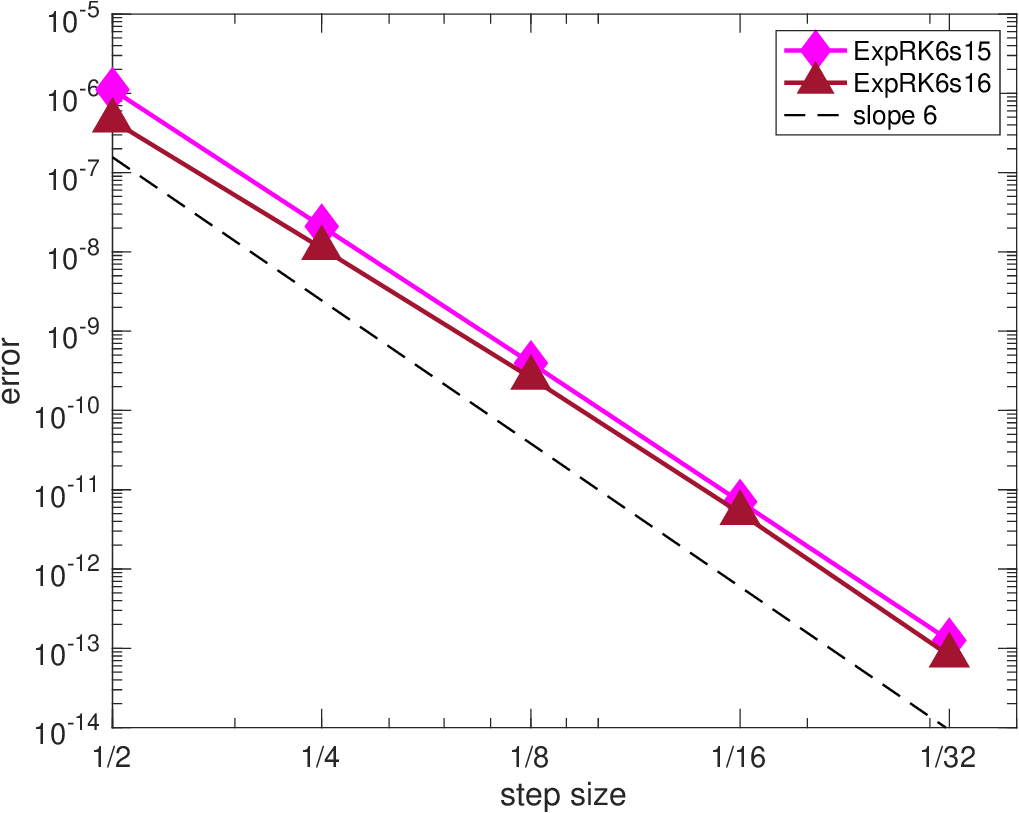,width=0.33\linewidth,clip=}
\hspace{4mm}
\epsfig{file=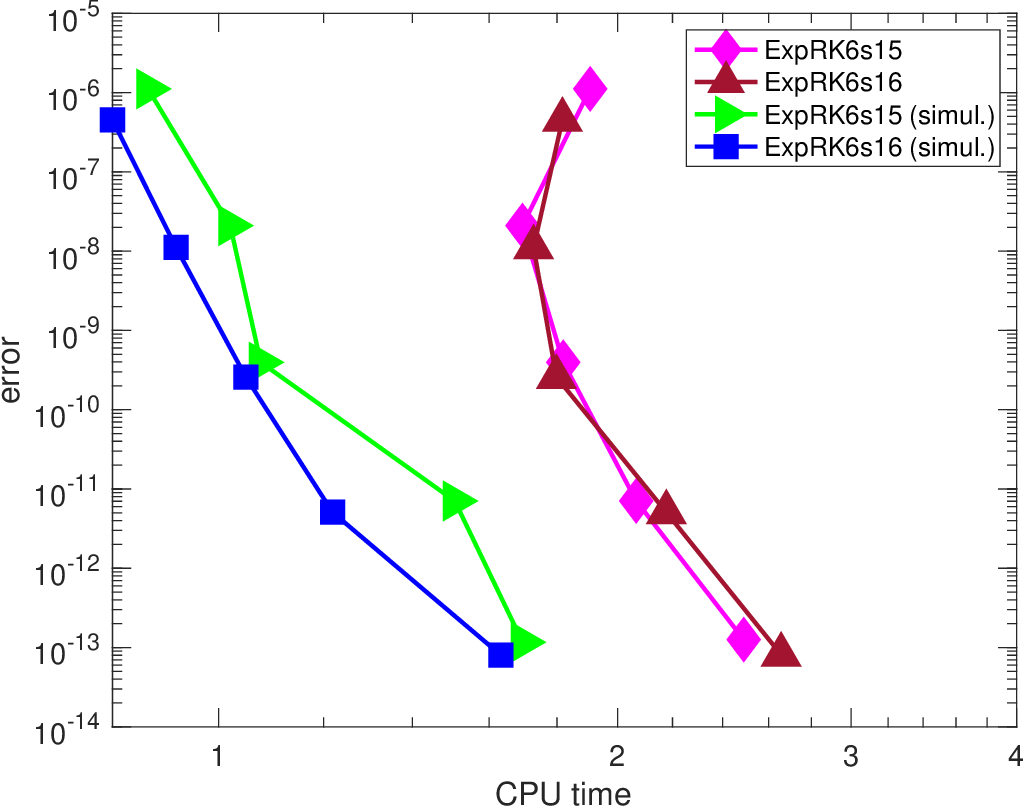,width=0.33\linewidth,clip=}  
\end{tabular}
\caption{\label{fig5.1} Order plots (left) and total CPU times (right) of $\mathtt{ExpRK6s15}$ and $\mathtt{ExpRK6s16}$ when applied to Example~\ref{ex1}. %The errors at time $t=1$ are plotted as functions of step sizes (left) and the total CPU time in second (right). 
%A straight line with slope 6 is also added to verify convergence rates.
}
\end{figure}

When considering the PDE in  Example~\ref{ex1} with inhomogeneous boundary conditions, our experiments show that, depending on each specific case, $\mathtt{ExpRK6s15}$ and  $\mathtt{ExpRK6s16}$ might or might not suffer from an order reduction. For instance, with the time-dependent boundary values $u(0) = t, \ u(1) = t + 1$   (e.g., corresponding to $u(x,t)=x(1-x)e^t + x + t$), we do not observe any order reduction. However, with  $u(0) = \cos(t), \ u(1) = \cos(t + 1)$  (e.g., corresponding to $u(x,t)=x(1-x)e^t + \cos(x + t)$, where the solution involves oscillating components), we observe significant order reduction in both 6th-order integrators. Interestingly, this phenomenon is not observed with lower-order ($\le 3$) ExpRK schemes which strictly satisfy all the stiff order conditions.
Since studying order reduction for ExpRK methods is not the primary focus of this work, we do not elaborate further details here, but instead refer to some recent results \cite{Cano2017,Cano2018,Cano2022a,Cano2022b}.
%%% bibliography (using bibtex file references.bib to input all references to the paper-----------------------------------------------------------
\bibliographystyle{elsarticle-num}
{\small
\bibliography{references}
}
\end{document}